\documentclass[11pt]{article}

\usepackage[margin=1.15in]{geometry}
\usepackage{amsmath,amssymb,amsthm,mathtools}
\usepackage{microtype}
\usepackage[numbers,sort&compress]{natbib}
\usepackage[colorlinks=true,allcolors=blue]{hyperref}
\hypersetup{
  pdftitle={Asymptotically Lossless Exact and Pivotal E-Values},
  pdfauthor={Hengzhi He and Guang Cheng}
}

\newtheorem{theorem}{Theorem}[section]
\newtheorem{proposition}[theorem]{Proposition}
\newtheorem{lemma}[theorem]{Lemma}

\theoremstyle{definition}
\newtheorem{definition}[theorem]{Definition}
\theoremstyle{remark}
\newtheorem{remark}[theorem]{Remark}

\newcommand{\cP}{\mathcal P}
\newcommand{\cX}{\mathcal X}
\newcommand{\cY}{\mathcal Y}
\newcommand{\E}{\mathbb E}
\newcommand{\R}{\mathbb R}
\newcommand{\law}{\mathcal L}
\newcommand{\one}{\mathbf 1}

\title{Asymptotically Lossless Exact and Pivotal E-Values}
\author{Hengzhi He\thanks{Department of Statistics and Data Science, University of California, Los Angeles.}
\and Guang Cheng\footnotemark[1]}
\date{August 2026}

\begin{document}
\maketitle

\begin{abstract}
Consider testing a finite composite null $\cP={P_1,\ldots,P_L}$ against a simple alternative $Q$ using $n$ i.i.d. observations. Let $\ell_n$ denote the supremum of the expected log e-value over e-variables that are exact under every $P_i$ and pivotal across the nulls. \citet{zhang2024existence} showed that $\ell_n$ is superadditive and asked whether $\ell_n/n$ converges to the upper bound $\min_i D(Q|P_i)$ under their joint atomlessness and absolute-continuity assumptions. We prove that it does. The proof draws on techniques developed by \citet{zhang2024existence} and by \citet{farooq2024matrix}.
\end{abstract}

\section{Introduction}

Let $P_1,\ldots,P_L$ be a finite collection of null distributions and let
$Q$ be a simple alternative.  A nonnegative statistic $E$ is an e-variable
for the null if $\E_{P_i}E\leq 1$ for every $i$ \citet{vovk2021values}.  Two additional properties
are natural when the null is composite \citet{zhang2024existence}.  Exactness requires
$\E_{P_i}E=1$ for every $i$, while pivotality requires the law of $E$ to be
the same under all $P_i$.  Note that these conditions are substantially stronger than
ordinary e-validity at a fixed sample size.

\citet{zhang2024existence} studied the existence and
optimal power of such statistics.  For $n$ i.i.d. observations, let
$\ell_n$ denote the supremum of $\E_{Q^{\otimes n}}\log E$ over all exact and
pivotal e-variables based on these observations.  Under their joint atomlessness and absolute-continuity
assumptions, and when $Q\notin\cP$, their Proposition~4.11 establishes
superadditivity and the upper bound
\[
  \lim_{n\to\infty}\frac{\ell_n}{n}
  \leq \min_{1\leq i\leq L}D(Q\|P_i),
\]
and asks whether equality always holds; see also their Gaussian example in
which the bound is attained \cite[Example~4.13]{zhang2024existence}. 

We prove a stronger result: whenever $Q\notin\cP$, the equality holds assuming only joint atomlessness of $(P_1,\ldots,P_L,Q)$; in particular, the absolute-continuity condition $P_i\ll Q$ imposed in Proposition~4.11 of \citet{zhang2024existence} is unnecessary.

The remainder of the paper is organized as follows. Section~2 introduces the setting, defines joint atomlessness and states the main result. Section~3 recalls the exact large-sample matrix-majorization theorem used in the proof. Section~4 constructs a suitable finite target experiment and verifies the conditions of that theorem. Section~5 extends the finite construction to the general setting and completes the proof.

\section{Setting and main result}

Let $\cX$ be a Polish space equipped with its Borel $\sigma$-field.  Throughout,
all logarithms are natural, and relative entropy is allowed to take the value
$+\infty$:
\[
  D(R\|S)=
  \begin{cases}
    \displaystyle\int \log\!\left(\frac{dR}{dS}\right)dR,
      & R\ll S,\\[1ex]
    +\infty,&\text{otherwise}.
  \end{cases}
\]

Fix $L\geq1$ and probability measures
$\cP=\{P_1,\ldots,P_L\}$ and $Q$ on $\cX$.
Repeated copies of the same null distribution are immaterial and will be
deleted.  We assume that $Q\notin\cP$.

For a measurable $E:\cX^n\to[0,\infty)$, call $E$ \emph{exact} if
\[
  \E_{P_i^{\otimes n}}E=1,\qquad i=1,\ldots,L,
\]
and \emph{pivotal} if
\[
  \law_{P_1^{\otimes n}}(E)=\cdots=
  \law_{P_L^{\otimes n}}(E).
\]
We use the convention $\log0=-\infty$.  An extended expectation below is said
to be defined when at least one of the expectations of the positive and
negative parts is finite.  Define the optimal exact-pivotal
e-power
\begin{equation}\label{eq:ell-definition}
  \ell_n:=\sup\left\{
    \E_{Q^{\otimes n}}\log E:
    E\text{ is exact and pivotal, and the extended expectation is defined}
  \right\}.
\end{equation}
We use a supremum so that no attainment assertion is needed.  The constant
e-variable $E\equiv1$ shows that $\ell_n\geq0$.

We recall the joint atomlessness condition from
\cite[Definition~2.2]{zhang2024existence}.

\begin{definition}[Joint atomlessness]\label{def:ja}
A finite tuple $(S_1,\ldots,S_d)$ of probability measures on $\cX$ is
\emph{jointly atomless} if there exist a probability measure $\mu$ dominating
every $S_j$ and a real-valued random variable $\xi$ such that, under $\mu$,
$\xi$ is atomless and independent of
\[
  \left(\frac{dS_1}{d\mu},\ldots,\frac{dS_d}{d\mu}\right).
\]
\end{definition}

The following Theorem is the main result.

\begin{theorem}
\label{thm:main}
Suppose that $(P_1,\ldots,P_L,Q)$ is jointly atomless and that
$Q\notin\{P_1,\ldots,P_L\}$.  Then, in the extended interval $[0,\infty]$,
\begin{equation}\label{eq:main-limit}
  \lim_{n\to\infty}\frac{\ell_n}{n}
  =\min_{1\leq i\leq L}D(Q\|P_i).
\end{equation}
In particular, the conclusion holds under the joint atomlessness and
absolute-continuity assumptions of Proposition~4.11 in
\cite{zhang2024existence}.
\end{theorem}

\section{Exact large-sample matrix majorization}
\label{sec:majorization}

We state the finite-experiment result used in the proof.  A finite experiment
is a tuple $\boldsymbol r=(r_1,\ldots,r_d)$ of probability vectors on a common
finite alphabet.  It majorizes another tuple
$\boldsymbol s=(s_1,\ldots,s_d)$ if there is one column-stochastic matrix $K$
such that $Kr_j=s_j$ for every $j$.  A finite experiment has \emph{full
support} if every entry of every one of its probability vectors is strictly
positive.  Tensor powers are taken componentwise, so
$\boldsymbol r^{\otimes m}=(r_1^{\otimes m},\ldots,r_d^{\otimes m})$.

For a Markov kernel $K$ and a probability measure $R$, we write
\[
 (K_{\#}R)(B)=\int K(x,B)\,R(dx)
\]
for the pushforward measure.

For $\alpha=(\alpha_1,\ldots,\alpha_d)$, put
\[
  \alpha_{\max}=\max_j\alpha_j,
  \qquad
  A=\left\{\alpha\in\R^d:\sum_j\alpha_j=1\right\},
\]
and define
\begin{align*}
  A_+&=\{\alpha\in A:\alpha_j\geq0\text{ for every }j\},\\
  A_-&=\bigcup_{k=1}^d
  \{\alpha\in A:\alpha_k\geq1,\ \alpha_j\leq0\text{ for }j\neq k\}.
\end{align*}
Let $e_j$ denotes the $j$th standard basis vector of $\R^d$.
For a full-support tuple $\boldsymbol r$, and
$\alpha\in(A_+\cup A_-)\setminus\{e_1,\ldots,e_d\}$, set (\citet{farooq2024matrix})
\begin{equation}\label{eq:matrix-alpha}
  \mathsf D_\alpha(\boldsymbol r)
  :=\frac{1}{\alpha_{\max}-1}
  \log\sum_x\prod_{j=1}^d r_j(x)^{\alpha_j}.
\end{equation}
For later use, for full-support probability vectors $R,S$, the ordinary
R\'enyi and max divergences are (\citet{6832827})
\[
 D_t(R\|S)=\frac{1}{t-1}\log\sum_xR(x)^tS(x)^{1-t},
 \qquad
 D_\infty(R\|S)=\log\max_x\frac{R(x)}{S(x)}.
\]
The first definition is for $t\in(0,1)\cup(1,\infty)$, with
$D_1(R\|S)=D(R\|S)$ by continuity.
Also for $\beta=(\beta_1,\ldots,\beta_d)\in\R^d$, define
\[
  \beta_{\max}=\max_j\beta_j,
  \qquad B=\left\{\beta\in\R^d:\sum_j\beta_j=0\right\},
\]
\[
  B_-=\bigcup_{k=1}^d
  \{\beta\in B:\beta_k\geq0,\ \beta_j\leq0\text{ for }j\neq k\}.
\]
For $\beta\in B_-\setminus\{0\}$, the tropical divergence (\citet{farooq2024matrix}) is
\begin{equation}\label{eq:matrix-tropical}
  \mathsf D^{\mathbb T}_\beta(\boldsymbol r)
  :=\frac{1}{\beta_{\max}}
  \log\max_x\prod_{j=1}^d r_j(x)^{\beta_j}.
\end{equation}

The following is part of Theorem~19 of \citet{farooq2024matrix}, specialized to
probability vectors with full support.

\begin{theorem}[Exact large-sample majorization]\label{thm:ffht}
Let $\boldsymbol r=(r_1,\ldots,r_d)$ and
$\boldsymbol s=(s_1,\ldots,s_d)$ be full-support finite experiments.  Suppose
that
\begin{align}
 \mathsf D_\alpha(\boldsymbol r)&>
 \mathsf D_\alpha(\boldsymbol s)
 &&\text{for all }\alpha\in(A_+\cup A_-)
                 \setminus\{e_1,\ldots,e_d\},\label{eq:ffht-alpha}\\
 \mathsf D^{\mathbb T}_\beta(\boldsymbol r)&>
 \mathsf D^{\mathbb T}_\beta(\boldsymbol s)
 &&\text{for all }\beta\in B_-\setminus\{0\},\label{eq:ffht-trop}\\
 D(r_k\|r_j)&>D(s_k\|s_j)
 &&\text{for all }k\neq j.\label{eq:ffht-kl}
\end{align}
Then, for every sufficiently large $m$, there is a column-stochastic matrix
$K_m$ such that
\[
  K_m r_j^{\otimes m}=s_j^{\otimes m},
  \qquad j=1,\ldots,d.
\]
\end{theorem}

Theorem~\ref{thm:ffht} is the key ingredient in the proof of our main result.

\section{The binary target experiment}
\label{sec:binary}

Let $\mathcal A$ be finite and let
\[
  \boldsymbol\mu=(\mu_1,\ldots,\mu_d)
\]
be a tuple of pairwise distinct, full-support probability vectors, where
$d=L+1$ and $\mu_d=q$ is the alternative.  Write
\begin{equation}\label{eq:rho-k}
  \rho_k:=\min_{j\neq k}D(\mu_k\|\mu_j)>0,
  \qquad
  r_0:=\rho_d=\min_{1\leq i\leq L}D(q\|\mu_i).
\end{equation}

Fix $c\in(0,r_0)$, and choose any positive sequence $(s_b)_{b\geq1}$ such that
\begin{equation}\label{eq:two-scales}
  s_b\longrightarrow\infty,
  \qquad \frac{s_b}{b}\longrightarrow0.
\end{equation}
For example, one may take $s_b=\sqrt b$.  Put
\begin{equation}\label{eq:binary-laws}
  a_b=e^{-bc},\qquad \delta_b=e^{-s_b},
  \qquad
  F_b=(a_b,1-a_b),\qquad
  G_b=(1-\delta_b,\delta_b),
\end{equation}
and define the target tuple
\[
  \boldsymbol R_b=(\underbrace{F_b,\ldots,F_b}_{L\text{ copies}},G_b).
\]

\begin{lemma}[Binary divergence bounds]\label{lem:binary-bounds}
For all sufficiently large $b$,
\begin{align}
 D_\infty(G_b\|F_b)&=bc+\log(1-\delta_b)<bc,
 \label{eq:forward-dinf}\\
 D_\infty(F_b\|G_b)&=s_b+\log(1-a_b)<s_b.
 \label{eq:reverse-dinf}
\end{align}
Consequently,
\[
  D(G_b\|F_b)<bc,
  \qquad D(F_b\|G_b)<s_b,
\]
and
\begin{equation}\label{eq:binary-rate}
  \frac{1}{b}D(G_b\|F_b)\longrightarrow c.
\end{equation}
\end{lemma}

\begin{proof}
For all sufficiently large $b$, $a_b+\delta_b<1$.  This inequality is
equivalent to
\[
 \frac{1-\delta_b}{a_b}>\frac{\delta_b}{1-a_b}
 \quad\text{and}\quad
 \frac{1-a_b}{\delta_b}>\frac{a_b}{1-\delta_b}.
\]
Thus $G_b/F_b$ is maximized at the first atom and the reverse likelihood
ratio is maximized at the second atom, giving
\eqref{eq:forward-dinf} and \eqref{eq:reverse-dinf}.  The KL bounds follow
from $D\leq D_\infty$.  Finally,
\begin{align*}
 D(G_b\|F_b)
 ={}&(1-\delta_b)\{bc+\log(1-\delta_b)\}\\
 &+\delta_b\{-s_b-\log(1-a_b)\}.
\end{align*}
Since $\delta_b\to0$, $s_b=o(b)$ and $a_b\to0$, division by $b$ proves
\eqref{eq:binary-rate}.
\end{proof}

We now verify the complete spectral criterion of Theorem~\ref{thm:ffht}.

\begin{proposition}[Exact degradation to the binary experiment]
\label{prop:finite-conversion}
For the fixed $c$ and sequence $(s_b)$ satisfying
\eqref{eq:two-scales}, there exists $b_0$ such that, for every $b\geq b_0$,
the source experiment
$\boldsymbol\mu^{\otimes b}$ and target experiment $\boldsymbol R_b$ satisfy
all strict inequalities
\eqref{eq:ffht-alpha}--\eqref{eq:ffht-kl}.  Consequently, for each such $b$
there exists $m_0(b)$ such that, for every $m\geq m_0(b)$, one Markov kernel
$K_{b,m}:\mathcal A^{bm}\rightsquigarrow\{0,1\}^m$ satisfies
\begin{equation}\label{eq:finite-conversion}
  (K_{b,m})_{\#}\mu_i^{\otimes bm}=F_b^{\otimes m}
  \quad(1\leq i\leq L),
  \qquad
  (K_{b,m})_{\#}q^{\otimes bm}=G_b^{\otimes m}.
\end{equation}
\end{proposition}

\begin{proof}
The proof is divided according to the three kinds of inequalities in
Theorem~\ref{thm:ffht}.

\smallskip
\noindent\emph{The $A_-$ region for \eqref{eq:ffht-alpha}.}
Fix $\alpha\in A_-\setminus\{e_1,\ldots,e_d\}$.  There is a unique index $k$
such that $\lambda:=\alpha_k>1$.  For $j\neq k$, set
\[
  w_j=-\frac{\alpha_j}{\lambda-1}.
\]
Then $w_j\geq0$ and $\sum_{j\neq k}w_j=1$.  With
\[
  Z_{k,w}(x)=\sum_{j\neq k}w_j
  \log\frac{\mu_k(x)}{\mu_j(x)},
\]
Jensen's inequality gives
\begin{equation}\label{eq:amin-source}
 \mathsf D_\alpha(\boldsymbol\mu)
 =\frac{1}{\lambda-1}\log
   \E_{\mu_k}e^{(\lambda-1)Z_{k,w}}
 \geq \E_{\mu_k}Z_{k,w}
 \geq\rho_k.
\end{equation}
If $k=d$, then $\alpha_d=\lambda>1$ is the exponent attached to the last
target distribution $G_b$, whereas every $\alpha_j$, $j<d$, is attached to
$F_b$.  Since all the first $d-1$ target distributions coincide and
$\sum_{j<d}\alpha_j=1-\lambda$, we obtain
\[
  \mathsf D_\alpha(\boldsymbol R_b)
  =D_\lambda(G_b\|F_b)
  \leq D_\infty(G_b\|F_b)<bc<br_0
  \leq\mathsf D_\alpha(\boldsymbol\mu^{\otimes b}).
\]

The equality follows from \eqref{eq:matrix-alpha} and the definition of $\boldsymbol R_b$, the first inequality from the monotonicity of R'enyi divergence in its order, and the bound $D_\infty(G_b|F_b)<bc$ from \eqref{eq:forward-dinf} in Lemma~\ref{lem:binary-bounds}; finally, the choice $c<r_0$, the identity $r_0=\rho_d$ in \eqref{eq:rho-k}, \eqref{eq:amin-source}, and additivity under tensor products give $bc<b r_0\leq\mathsf D_\alpha(\boldsymbol\mu^{\otimes b})$.

If $k<d$, put $u=-\alpha_d\in[0,\lambda-1]$. Only this weight on the last column matters. Indeed, since the first $d-1$ target columns are all equal to $F_b$, their exponents enter only through their sum $\sum_{j<d}\alpha_j=1+u$, while the last column $G_b$ has exponent $\alpha_d=-u$. We then have \[ \mathsf D_\alpha(\boldsymbol R_b) =\begin{cases} \dfrac{u}{\lambda-1}D_{1+u}(F_b\|G_b),&u>0,\\[1ex] 0,&u=0, \end{cases} \leq D_\infty(F_b\|G_b)<s_b. \] 

(The first inequality uses $u/(\lambda-1)\leq1$ and the monotonicity of R'enyi divergence in its order, while the second follows from \eqref{eq:reverse-dinf} in Lemma~\ref{lem:binary-bounds}.)

Since $s_b/b\to0$ by \eqref{eq:two-scales} and $\rho_k>0$ by
\eqref{eq:rho-k}, we have $s_b<b\rho_k$ for all sufficiently large $b$.
Moreover, \eqref{eq:amin-source} and additivity under tensor products imply
\[
  \mathsf D_\alpha(\boldsymbol\mu^{\otimes b})
  \geq b\rho_k.
\]
Combining these facts with the preceding bound
$\mathsf D_\alpha(\boldsymbol R_b)<s_b$ yields
\[
  \mathsf D_\alpha(\boldsymbol R_b)
  <\mathsf D_\alpha(\boldsymbol\mu^{\otimes b})
\]
for all sufficiently large $b$.  These bounds are independent of $\alpha$
for each fixed $k$.  Since there are only finitely many possible indices
$k$, and the case $k=d$ was handled above, a common sufficiently large $b$
works for all
$\alpha\in A_-\setminus\{e_1,\ldots,e_d\}$.

\smallskip
\noindent\emph{The $A_+$ region for \eqref{eq:ffht-alpha}.}
For each $k$, let $\mathcal C_k$ be the set of pairs $(\varepsilon,w)$
such that
\begin{align*}
 &0\leq\varepsilon\leq1-1/d,
 \qquad w\in\Delta(\{1,\ldots,d\}\setminus\{k\}),\\
 &\varepsilon w_j\leq1-\varepsilon\quad\text{for every }j\neq k.
\end{align*}
Here $\Delta(I)$ denotes the probability simplex on the finite set $I$.
The set $\mathcal C_k$ is compact.  Define
\begin{equation}\label{eq:phi-definition}
 \Phi_k(\varepsilon,w)=
 \begin{cases}
 -\dfrac{1}{\varepsilon}
  \log\E_{\mu_k}e^{-\varepsilon Z_{k,w}},&\varepsilon>0,\\[2ex]
 \displaystyle\sum_{j\neq k}w_jD(\mu_k\|\mu_j),&\varepsilon=0.
 \end{cases}
\end{equation}

Since the alphabet $\mathcal A$ is finite and every probability vector $\mu_j$
has full support, there exists a constant $M<\infty$ such that
\begin{equation*}
 \left|\log\frac{\mu_k(x)}{\mu_j(x)}\right|\leq M
\end{equation*}
for every $x\in\mathcal A$ and every $j\neq k$. Since $w$ is a probability
vector, it follows that
\begin{equation*}
 |Z_{k,w}(x)|
 \leq
 \sum_{j\neq k}w_j
 \left|\log\frac{\mu_k(x)}{\mu_j(x)}\right|
 \leq M
\end{equation*}
uniformly over $x$ and $w$. Consequently, as $\varepsilon\downarrow0$,
the Taylor expansion
\begin{equation*}
 \log\E_{\mu_k}e^{-\varepsilon Z_{k,w}}
 =
 -\varepsilon\E_{\mu_k}Z_{k,w}
 +O(\varepsilon^2)
\end{equation*}
holds with a remainder uniform over the simplex of $w$. Moreover,
\begin{equation*}
 \E_{\mu_k}Z_{k,w}
 =
 \sum_{j\neq k}w_j
 \E_{\mu_k}\log\frac{\mu_k}{\mu_j}
 =
 \sum_{j\neq k}w_jD(\mu_k\|\mu_j).
\end{equation*}
Therefore,
\begin{equation*}
 -\frac{1}{\varepsilon}
 \log\E_{\mu_k}e^{-\varepsilon Z_{k,w}}
 =
 \sum_{j\neq k}w_jD(\mu_k\|\mu_j)
 +O(\varepsilon)
\end{equation*}
uniformly in $w$. Thus the value assigned to $\Phi_k(0,w)$ in
\eqref{eq:phi-definition} is precisely the continuous extension of the
expression for $\varepsilon>0$. Since continuity for $\varepsilon>0$
follows directly from the finite-sum representation, $\Phi_k$ is continuous
on $\mathcal C_k$.

At $\varepsilon=0$, the definition of $\rho_k$ in \eqref{eq:rho-k} gives
\begin{equation*}
 \Phi_k(0,w)
 =
 \sum_{j\neq k}w_jD(\mu_k\|\mu_j)
 \geq
 \rho_k>0.
\end{equation*}
Now suppose that $\varepsilon>0$. The exponents $1-\varepsilon$ and
$\varepsilon w_j$, $j\neq k$, are nonnegative and sum to one. Generalized
H\"older's inequality therefore gives
\begin{equation*}
 \sum_x\mu_k(x)^{1-\varepsilon}
       \prod_{j\neq k}\mu_j(x)^{\varepsilon w_j}
 \leq
 \left(\sum_x\mu_k(x)\right)^{1-\varepsilon}
 \prod_{j\neq k}
 \left(\sum_x\mu_j(x)\right)^{\varepsilon w_j}
 =1.
\end{equation*}
This inequality is strict. Indeed, the definition of $\mathcal C_k$ implies
that
\begin{equation*}
 1-\varepsilon\geq\frac{1}{d}>0,
\end{equation*}
and at least one $w_j$ is positive because
$\sum_{j\neq k}w_j=1$. Equality in generalized H\"older's inequality would
therefore require $\mu_k$ to be proportional to every $\mu_j$ for which
$w_j>0$. Since these are probability vectors, proportionality would imply
$\mu_k=\mu_j$, contradicting the assumed pairwise distinctness of
$\mu_1,\ldots,\mu_d$.
Consequently,
\begin{equation*}
 \E_{\mu_k}e^{-\varepsilon Z_{k,w}}
 =
 \sum_x\mu_k(x)^{1-\varepsilon}
       \prod_{j\neq k}\mu_j(x)^{\varepsilon w_j}
 <1,
\end{equation*}
and hence $\Phi_k(\varepsilon,w)>0$.

Thus $\Phi_k$ is continuous and strictly positive on the compact set
$\mathcal C_k$, so it attains a strictly positive minimum there. Since there
are only finitely many indices $k$, we obtain
\begin{equation}\label{eq:eta-positive}
 \eta:=
 \min_{1\leq k\leq d}
 \min_{(\varepsilon,w)\in\mathcal C_k}
 \Phi_k(\varepsilon,w)>0.
\end{equation}

Now take $\alpha\in A_+\setminus\{e_1,\ldots,e_d\}$ and choose a coordinate
attaining $\alpha_{\max}$.  If a null coordinate $k<d$ attains the maximum,
write $\varepsilon=1-\alpha_k$ and $t=\alpha_d$.  Then $0\leq t\leq\varepsilon$.
With
\[
 H_b(t)=\sum_yF_b(y)^{1-t}G_b(y)^t,
\]
we have
\[
 \mathsf D_\alpha(\boldsymbol R_b)
 =-\frac{1}{\varepsilon}\log H_b(t).
\]
The function $h_b(t)=\log H_b(t)$ is convex,
$h_b(0)=0$ and $h_b'(0)=-D(F_b\|G_b)$.  Its tangent inequality therefore
implies
\begin{equation}\label{eq:aplus-null-target}
 \mathsf D_\alpha(\boldsymbol R_b)
 \leq\frac{t}{\varepsilon}D(F_b\|G_b)<s_b.
\end{equation}
On the other hand, the parametrization
$\alpha_k=1-\varepsilon$ and
$\alpha_j=\varepsilon w_j$, $j\neq k$, gives
\begin{align}
\label{repara_modification}
 \mathsf D_\alpha(\boldsymbol\mu)
 =
 \Phi_k(\varepsilon,w)
 \geq\eta.
\end{align}
By additivity under tensor products,
\[
 \mathsf D_\alpha(\boldsymbol\mu^{\otimes b})
 \geq b\eta.
\]
Since $s_b/b\to0$ by \eqref{eq:two-scales} and $\eta>0$ by
\eqref{eq:eta-positive}, we have $s_b<b\eta$ for all sufficiently large
$b$. Combining this with \eqref{eq:aplus-null-target} yields
\[
 \mathsf D_\alpha(\boldsymbol R_b)
 <
 \mathsf D_\alpha(\boldsymbol\mu^{\otimes b})
\]
for all sufficiently large $b$, uniformly over all $\alpha$ for which a
null coordinate attains $\alpha_{\max}$.

It remains to consider the case in which the alternative coordinate $d$ is
the unique maximizer.  Put $t=\alpha_d$, $\varepsilon=1-t$, and
$w_i=\alpha_i/\varepsilon$.  Since $\alpha\neq e_d$, we have
$\varepsilon>0$.  Then
\begin{equation}\label{eq:aplus-q-identities}
 \mathsf D_\alpha(\boldsymbol R_b)=D_t(G_b\|F_b),
 \qquad
 \mathsf D_\alpha(\boldsymbol\mu)=\Phi_d(\varepsilon,w).
\end{equation}
Choose $c_1\in(c,r_0)$.  Uniform continuity at $\varepsilon=0$ and
\[
 \Phi_d(0,w)=\sum_{i=1}^Lw_iD(q\|\mu_i)\geq r_0
\]
give a $\kappa>0$ such that
$\Phi_d(\varepsilon,w)\geq c_1$ whenever
$\varepsilon\leq\kappa$.  In that region, Jensen's inequality gives
\[
 \mathsf D_\alpha(\boldsymbol R_b)
 =D_t(G_b\|F_b)<bc<bc_1
 \leq b\Phi_d(\varepsilon,w)
 =\mathsf D_\alpha(\boldsymbol\mu^{\otimes b}),
\]
where the monotonicity of R\'enyi divergence in its order gives
$D_t(G_b\|F_b)\leq D(G_b\|F_b)$ for $t<1$, and
Lemma~\ref{lem:binary-bounds} bounds the latter by $bc$.

When $\varepsilon\geq\kappa$, retaining only the second term in the binary
Hellinger sum gives
\begin{align}
 D_t(G_b\|F_b)
 &=-\frac{1}{1-t}\log
    \sum_yG_b(y)^tF_b(y)^{1-t}\notag\\
 &\leq-\log(1-a_b)+\frac{t}{1-t}s_b\notag\\
 &\leq-\log(1-a_b)+\frac{1-\kappa}{\kappa}s_b=o(b),
 \label{eq:aplus-away}
\end{align}
uniformly over the region. 

By \eqref{repara_modification} and additivity under tensor products,
\[
 \mathsf D_\alpha(\boldsymbol\mu^{\otimes b})
 =
 b\mathsf D_\alpha(\boldsymbol\mu)
 \geq b\eta.
\]
Combining this lower bound with the uniform $o(b)$ upper bound in
\eqref{eq:aplus-away}, we obtain
\[
 \mathsf D_\alpha(\boldsymbol R_b)
 <
 b\eta
 \leq
 \mathsf D_\alpha(\boldsymbol\mu^{\otimes b})
\]
for all sufficiently large $b$, uniformly over the region
$\varepsilon\geq\kappa$ in which $\alpha_d$ is the unique maximizer. This
completes the verification of \eqref{eq:ffht-alpha} over the $A_+$ region.
If $\alpha_d$ is tied with a null coordinate $k<d$ for the maximum, then
that null coordinate also attains $\alpha_{\max}$, so this case was already
covered by the preceding null-maximum argument.

\smallskip
\noindent\emph{Proof of \eqref{eq:ffht-trop}}
The tropical divergence is invariant under positive rescaling of $\beta$, so
we may normalize $\beta_{\max}=1$.  There is then a unique positive
coordinate $\beta_k=1$; put $w_j=-\beta_j$ for $j\neq k$.  Then
\begin{equation}\label{eq:trop-source}
 \mathsf D^{\mathbb T}_\beta(\boldsymbol\mu)
 =\max_x Z_{k,w}(x)
 \geq\E_{\mu_k}Z_{k,w}
 \geq\rho_k.
\end{equation}
For the target experiment $\boldsymbol R_b$, 
\eqref{eq:forward-dinf} and \eqref{eq:reverse-dinf} give
\[
 \mathsf D^{\mathbb T}_\beta(\boldsymbol R_b)
 =\begin{cases}
   D_\infty(G_b\|F_b)<bc,&k=d,\\
   w_dD_\infty(F_b\|G_b)<s_b,&k<d.
 \end{cases}
\]

Combining the preceding bounds for $\boldsymbol R_b$ with
\eqref{eq:trop-source} and additivity under tensor products, and using the
choice $c<r_0$, the identity $r_0=\rho_d$ in \eqref{eq:rho-k}, and
$s_b/b\to0$ from \eqref{eq:two-scales}, we obtain a common $b_0$ such that
\begin{equation*}
  \mathsf D^{\mathbb T}_\beta(\boldsymbol\mu^{\otimes b})
  >
  \mathsf D^{\mathbb T}_\beta(\boldsymbol R_b)
\end{equation*}
for every $b\geq b_0$ and every $\beta\in B_-\setminus\{0\}$.  This verifies
the strict tropical inequality \eqref{eq:ffht-trop}.

\smallskip
\noindent\emph{Proof of \eqref{eq:ffht-kl}.}
If $k,j<d$, the corresponding target columns coincide, and hence
\[
 D((\boldsymbol R_b)_k\|(\boldsymbol R_b)_j)=0
 <bD(\mu_k\|\mu_j)= 
 D(\mu_k^{\otimes b}\|\mu_j^{\otimes b}).
\]
For $i<d$, Lemma~\ref{lem:binary-bounds} gives
\[
 D(G_b\|F_b)<bc<br_0\leq bD(q\|\mu_i)=D(q^{\otimes b}|\mu_i^{\otimes b})
\]
and, for all large $b$ (since $\frac{s_b}{b} \to 0 $),
\[
 D(F_b\|G_b)<s_b<bD(\mu_i\|q) =D(\mu_i^{\otimes b}\|q^{\otimes b}).
\]
Thus all ordered KL inequalities are strict.

We have verified every condition of Theorem~\ref{thm:ffht} for all large $b$.
Applying that theorem to source $\boldsymbol\mu^{\otimes b}$, target
$\boldsymbol R_b$, and repetition number $m$ gives
\eqref{eq:finite-conversion}, because
$(\mu_j^{\otimes b})^{\otimes m}=\mu_j^{\otimes bm}$.
\end{proof}

\begin{remark}[Randomization at the finite stage]
Proposition~\ref{prop:finite-conversion} produces a common Markov kernel.  It
does not assert that a deterministic transformation exists on the discrete
source alphabet.  The deterministic exact-pivotal statistic will be obtained
only after returning to the jointly atomless original experiment in
Section~\ref{sec:general}.
\end{remark}

\section{Reduction and purification}
\label{sec:general}

We first reduce an arbitrary finite family of measures to a pairwise distinct
full-support finite experiment without losing a prescribed strict KL margin.

\begin{lemma}[Finite full-support coarsening]\label{lem:coarsening}
Let $S_0,S_1,\ldots,S_L$ be pairwise distinct probability measures on the
Polish space $\cX$ and let
\[
 c<\min_{1\leq i\leq L}D(S_0\|S_i).
\]
There exist a finite alphabet $\mathcal A$ and a Markov kernel
$K:\cX\rightsquigarrow\mathcal A$ such that the probability vectors
$s_j=K_{\#}S_j$, $0\leq j\leq L$, are pairwise distinct, have full support,
and satisfy
\[
 \min_{1\leq i\leq L}D(s_0\|s_i)>c.
\]
\end{lemma}

\begin{proof}
Choose $c'>c$ below $\min_iD(S_0\|S_i)$.  We use the finite-partition
representation \cite[Section~7.1, equation~(7.1), and Lemma~7.4]{Gray2011}
\begin{equation}\label{eq:partition-kl}
 D(R\|S)=\sup_{\pi}D(R_\pi\|S_\pi),
\end{equation}
where the supremum is over finite measurable partitions. 

For each $i$, choose a finite partition $\pi_i$ whose discrete divergence is
larger than $c'$.  For every pair $j\neq k$, choose a measurable set on which
$S_j$ and $S_k$ differ.  Let $\pi=\{A_1,\ldots,A_M\}$ be the common refinement
of all these partitions and binary separating partitions.  Refinement cannot
decrease relative entropy, and the vectors
\[
 s_j^\pi=(S_j(A_1),\ldots,S_j(A_M))
\]
are pairwise distinct.

Let $u$ be uniform on $\{1,\ldots,M\}$ and set
\[
 s_j^{\pi,\varepsilon}=(1-\varepsilon)s_j^\pi+\varepsilon u.
\]
These vectors have full support and remain distinct because their pairwise
differences are multiplied by $1-\varepsilon$.  Moreover,
\[
 D(s_0^{\pi,\varepsilon}\|s_i^{\pi,\varepsilon})
 \longrightarrow D(s_0^\pi\|s_i^\pi)
\]
as $\varepsilon\downarrow0$, in the extended-real sense.  (This is immediate
term by term; if a positive numerator meets a zero denominator, the
corresponding term diverges to infinity.)  A common sufficiently small
$\varepsilon$ preserves all strict inequalities.  Finally, take
\[
 K(x,\{a\})=(1-\varepsilon)\one_{\{x\in A_a\}}
             +\frac{\varepsilon}{M},
 \qquad a=1,\ldots,M.
\] finishes the proof.
\end{proof}

\begin{lemma}[Joint atomlessness under products]\label{lem:ja-products}
If $(S_1,\ldots,S_d)$ is a jointly atomless tuple on $\cX$, then
$(S_1^{\otimes N},\ldots,S_d^{\otimes N})$ is jointly atomless for every
$N\geq1$.
\end{lemma}

\begin{proof}
Let $\mu$ and $\xi$ witness joint atomlessness and write
$r_j=dS_j/d\mu$.  The measure $\mu^{\otimes N}$ dominates every
$S_j^{\otimes N}$, with density
\[
  \frac{dS_j^{\otimes N}}{d\mu^{\otimes N}}(x_1,\ldots,x_N)
  =\prod_{t=1}^Nr_j(x_t).
\]
Under $\mu^{\otimes N}$, the variable $\xi(x_1)$ remains atomless and is
independent of
\[
 (r_1(x_1),\ldots,r_d(x_1))
\]
by the joint atomlessness assumption. Moreover, the remaining coordinates
$(x_2,\ldots,x_N)$ are independent of
\[
 \bigl(\xi(x_1),(r_1(x_1),\ldots,r_d(x_1))\bigr)
\]
under the product measure $\mu^{\otimes N}$.  Hence $\xi(x_1)$ is
independent of
\[
 \bigl((r_j(x_t))_{j=1}^d\bigr)_{t=1}^N
\]
and therefore of the displayed vector of product densities.
\end{proof}

We next record the exact consequence of simultaneous transport used below.

\begin{lemma}[Simultaneous purification]\label{lem:purification}
Let $\boldsymbol S=(S_1,\ldots,S_d)$ be a jointly atomless tuple on a Polish
space $\cX$, and let $\boldsymbol H=(H_1,\ldots,H_d)$ be a tuple on a Polish
space $\cY$.  If one Markov kernel sends every $S_j$ to $H_j$, then one
measurable map $T:\cX\to\cY$ sends every $S_j$ to $H_j$.
\end{lemma}

\begin{proof}
Set $\overline S=d^{-1}\sum_{j=1}^dS_j$ and
$\overline H=d^{-1}\sum_{j=1}^dH_j$.  Each $S_j$ is absolutely continuous
with respect to $\overline S$, and each $H_j$ is absolutely continuous with
respect to $\overline H$.  By Proposition~2.3(i) of
\citet{zhang2024existence}, the assumed simultaneous kernel is equivalent to
the convex-order relation
\[
 \left(\frac{dS_1}{d\overline S},\ldots,
       \frac{dS_d}{d\overline S}\right)\Bigg|_{\overline S}
 \succeq_{\mathrm{cx}}
 \left(\frac{dH_1}{d\overline H},\ldots,
       \frac{dH_d}{d\overline H}\right)\Bigg|_{\overline H}.
\]
Here the vertical bar denotes the law under the indicated mixture, and
$\succeq_{\mathrm{cx}}$ denotes domination in multivariate convex order.
Because $\boldsymbol S$ is jointly atomless, Proposition~2.3(ii) of the same
paper says that this convex-order relation is also equivalent to the
existence of a simultaneous measurable map.  
\end{proof}

Now we are ready to prove the main theorem.

\begin{proof}[Proof of Theorem~\ref{thm:main}]
Deleting duplicate nulls changes neither the exactness nor the pivotality
constraints, and every subtuple of a jointly atomless tuple is jointly
atomless.  We may therefore assume that
$P_1,\ldots,P_L,Q$ are pairwise distinct.  Put
\[
 R=\min_{1\leq i\leq L}D(Q\|P_i).
\]
Since relative entropy vanishes only for identical probability measures,
$R\in(0,+\infty]$.  Fix a finite $r$ with $0<r<R$, and choose a finite $c$
with $r<c<R$.  Apply
Lemma~\ref{lem:coarsening} with $S_0=Q$ and $S_i=P_i$.  We obtain one
single-observation kernel $M$ and pairwise distinct full-support vectors
\[
 q=M_{\#}Q,\qquad p_i=M_{\#}P_i,
 \qquad \min_iD(q\|p_i)>c.
\]

Apply the construction of Section~\ref{sec:binary} to
$(p_1,\ldots,p_L,q)$ with the above $c$ and, for definiteness,
$s_b=\sqrt b$.  Proposition~\ref{prop:finite-conversion} applies.  Choose a
sufficiently large $b$ so that all spectral conditions hold and, by
\eqref{eq:binary-rate},
\begin{equation}\label{eq:choose-b}
 \frac{1}{b}D(G_b\|F_b)>r.
\end{equation}
For any $m\geq m_0(b)$, compose the product kernel $M^{\otimes bm}$ with the
kernel in \eqref{eq:finite-conversion}.  This gives one simultaneous kernel
from the original experiment satisfying
\begin{equation}\label{eq:composed-kernel}
 P_i^{\otimes bm}\longmapsto F_b^{\otimes m}\quad(1\leq i\leq L),
 \qquad
 Q^{\otimes bm}\longmapsto G_b^{\otimes m}.
\end{equation}

By Lemma~\ref{lem:ja-products}, the source tuple in
\eqref{eq:composed-kernel} is jointly atomless.  Lemma~\ref{lem:purification}
therefore provides a statistic
$Y_{b,m}:\cX^{bm}\to\{0,1\}^m$ having exactly the target laws in
\eqref{eq:composed-kernel}.  Define
\begin{equation}\label{eq:constructed-e}
 E_{b,m}:=
 \frac{dG_b^{\otimes m}}{dF_b^{\otimes m}}(Y_{b,m}).
\end{equation}
Under every null, $Y_{b,m}$ has the same law $F_b^{\otimes m}$, and hence
$E_{b,m}$ is pivotal.  It is exact because
\[
 \E_{P_i^{\otimes bm}}E_{b,m}
 =\int\frac{dG_b^{\otimes m}}{dF_b^{\otimes m}}\,dF_b^{\otimes m}=1.
\]
Under the alternative,
\[
 \E_{Q^{\otimes bm}}\log E_{b,m}
 =D(G_b^{\otimes m}\|F_b^{\otimes m})
 =mD(G_b\|F_b).
\]
Together with \eqref{eq:choose-b}, this proves
\begin{equation}\label{eq:block-lower}
 \frac{\ell_{bm}}{bm}>r.
\end{equation}

For completeness, $\ell_n$ is superadditive without any additional
assumption.  Indeed, let $E_n$ and $E_m$ be exact and pivotal e-variables based on
independent blocks of $n$ and $m$ observations, respectively, then we define
\[
 E_{n+m}(x_1,\ldots,x_{n+m})
 :=E_n(x_1,\ldots,x_n)E_m(x_{n+1},\ldots,x_{n+m}),
\]
this is exact because the two factors are independent under each null, and is
pivotal because its null law is the product of two laws that do not depend on
the null index.  For finite $u<\ell_n$ and $v<\ell_m$, choose admissible
statistics whose e-powers exceed $u$ and $v$, respectively.  These e-powers
are not $-\infty$, so the negative parts of their log e-variables are
integrable and the alternative expected logarithms can be added safely.  Hence
$\ell_{n+m}>u+v$.  Letting $u\uparrow\ell_n$ and $v\uparrow\ell_m$ proves
$\ell_{n+m}\geq\ell_n+\ell_m$, including when one of the suprema is infinite.
Fekete's lemma (see, e.g., Lemma~1.2.1 of
\citet{steele1997probability})) in the extended-real form therefore yields
\[
 \lim_{n\to\infty}\frac{\ell_n}{n}
 =\sup_{n\geq1}\frac{\ell_n}{n}.
\]
Thus the block bound \eqref{eq:block-lower} implies that the limit is at least
$r$.  Since every $r\in(0,R)$ is allowed, the lower bound is $R$. 

For the matching upper bound, let $E$ be any exact and pivotal e-variable
based on $n$ observations.  For every $i\in\{1,\ldots,L\}$, the entropy
variational inequality gives
\[
  \E_{Q^{\otimes n}}\log E
  \leq
  D(Q^{\otimes n}\|P_i^{\otimes n})
  +\log\E_{P_i^{\otimes n}}E.
\]
Since $E$ is exact,
\[
  \E_{P_i^{\otimes n}}E=1.
\]
Moreover, relative entropy is additive under tensor products, so
\[
  D(Q^{\otimes n}\|P_i^{\otimes n})
  =nD(Q\|P_i).
\]
Consequently,
\[
  \E_{Q^{\otimes n}}\log E
  \leq nD(Q\|P_i).
\]
This inequality is understood in the extended-real sense.  If
$D(Q\|P_i)=+\infty$, it is trivial; otherwise, the finiteness of the
divergence implies $Q\ll P_i$.  Thus no absolute continuity assumption is
imposed.

Since the preceding inequality holds for every $i$, taking the minimum over
$i$ and then the supremum over all exact and pivotal e-variables gives
\[
  \ell_n
  \leq
  n\min_{1\leq i\leq L}D(Q\|P_i)
  =nR.
\]
Therefore,
\[
  \limsup_{n\to\infty}\frac{\ell_n}{n}\leq R.
\]
Combining this with the lower bound proves
\[
  \lim_{n\to\infty}\frac{\ell_n}{n}=R.
\]
\end{proof}

\section{Discussion}

Theorem~\ref{thm:main} shows that imposing exactness and pivotality does not
reduce the optimal first-order e-power.  More precisely, the largest expected
logarithmic growth rate attainable by an exact and pivotal e-variable
converges to
$
\min_{1\leq i\leq L}D(Q|P_i).
$

More precisely,
\[
 \lim_{n\to\infty}\frac{\ell_n}{n}
 =
 \min_{1\leq i\leq L}D(Q\|P_i).
\]

This answers in the affirmative the asymptotic question raised after
Proposition~4.11 of \citet{zhang2024existence}.  Moreover, the absolute
continuity assumption $P_i\ll Q$ imposed there is not needed for this
conclusion; joint atomlessness alone suffices.

The limiting rate has a natural information-theoretic interpretation.  If
the null distribution were known to be $P_i$, the optimal expected
logarithmic growth rate against $Q$ would be $D(Q|P_i)$.  Under the finite
composite null, no e-variable can exceed the smallest of these rates.
Theorem~\ref{thm:main} shows that this universal upper bound remains
asymptotically attainable even when the e-variable must be exact under every
null distribution and have the same null distribution for every $P_i$.
Thus exactness and pivotality may restrict finite-sample constructions, but
they do not change the asymptotic information rate.
\footnotesize
\setlength{\bibsep}{1pt}
\bibliographystyle{plainnat}
\bibliography{references}

\end{document}